\documentclass[]{elsarticle}
\usepackage{url}
\usepackage{setspace}
\usepackage{graphicx}
\usepackage{comment}
\usepackage{amsmath,amssymb,amsthm}
\usepackage{graphicx,tikz}

\newtheorem*{thm}{Theorem}

\newtheorem*{lemma}{Lemma}

\theoremstyle{definition}

\theoremstyle{remark}

\DeclareMathOperator{\spasn}{span}
\newcommand {\myvec}[1] {{\mbox{\boldmath $#1$}}}
\newcommand{\rev}{\color{black}}

\begin{document}
	\begin{frontmatter}
		\title{Randomly Aggregated Least Squares\\ for Support Recovery}

		\author{Ofir Lindenbaum %
			\fnref{fn1}}
		\ead{ofir.lindenbaum@yale.edu}
		
		\author{Stefan Steinerberger%
			\fnref{fn2} \corref{cor1}}
		\ead{steinerb@uw.edu}
		
		\fntext[fn1]{Program in Applied Mathematics, Yale University, New Haven, CT 06511, USA}
		\fntext[fn2]{Department of Mathematics, University of Washington, Seattle, WA 98195, USA }
		\cortext[cor1]{The work was funded by NSFDMS-1763179 and the Alfred P. Sloan Foundation. }
		
		%\footnote{ The work was funded by NSFDMS-1763179 and the Alfred P. Sloan Foundation. }

		%\maketitle
		
		\begin{abstract}
			We study the problem of exact support recovery: given an (unknown) vector $\myvec{\theta}^* \in \left\{-1,0,1\right\}^D$ with known sparsity $k = \|\myvec{\theta}^*\|_0$, we are given access to the noisy measurement 
			$$ \myvec{y} = \myvec{X}\myvec{\theta}^* + \myvec{\omega},$$
			where $\myvec{X} \in \mathbb{R}^{N \times D}$ is a (known) Gaussian matrix and the noise $\myvec{\omega} \in \mathbb{R}^N$ is an (unknown) Gaussian vector. How small can $N$ be for reliable recovery of the support of $\myvec{\theta}^*$? We present RAWLS (\textbf{R}andomly \textbf{A}ggregated un\textbf{W}eighted \textbf{L}east Squares \textbf{S}upport Recovery): the main idea is to take random subsets of the $N$ equations, perform least squares over this reduced bit of information, and average over many random subsets. We show that the proposed procedure can provably recover an approximation of $\myvec{\theta}^*$ and demonstrate its use through numerical examples. We use numerical simulations to demonstrate that the proposed procedure is beneficial for the task of support recovery. Finally, we observe that RAWLS is at par with several strong baselines in the low information regime (i.e. $N$ is small or $k$ is large).
		\end{abstract}

		
	\end{frontmatter}

	\section{Introduction}

	\begin{figure}[h!]\label{example}
		\begin{center}
			\begin{tikzpicture}[scale=0.5]
				\node at (-7,0) {$\myvec{y}=$};
				\node at (0,0) {\includegraphics[scale=0.5]{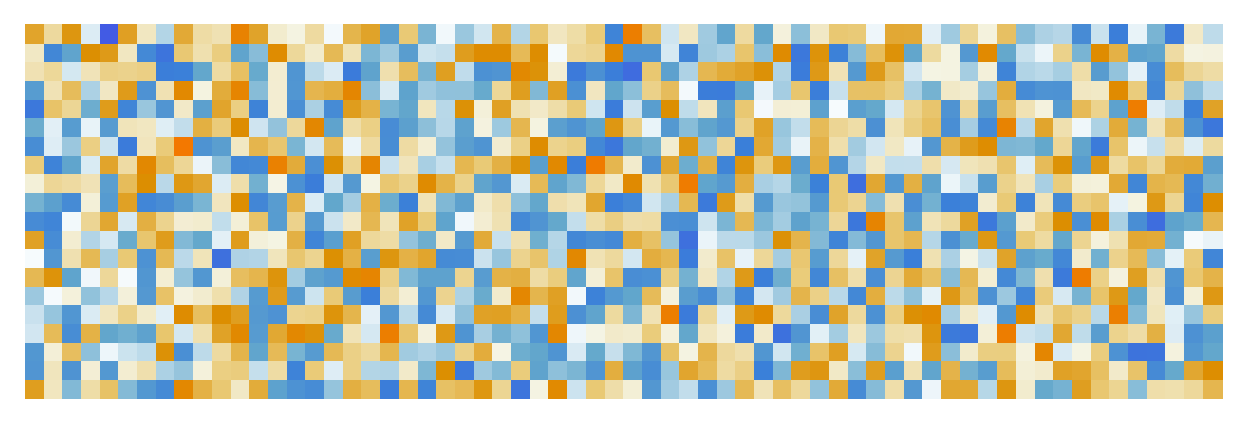}};
				\node at (7.35,0) {$\myvec{\theta }^*+ \myvec{w}$};
			\end{tikzpicture}
			\caption{We try to recover the support of $\myvec{\theta}^*$ from the observations $\myvec{X}$ and $\myvec{y}$, where $\myvec{y}=\myvec{X\theta}^* + \myvec{\omega}$. The (known) matrix $\myvec{X}$ is a Gaussian random matrix, so is the (unknown) noise $\myvec{\omega}$, we try to recover the support of $\myvec{\theta}^*$ with few measurements.}
		\end{center}
		
	\end{figure}

	The problem of support recovery, plays an important role in machine learning, signal processing, bioinformatics, and high dimensional statistics. In some applications, identifying the support leads to direct benefits such as reduction of memory and computational costs \cite{survey}, identification of cancer risk genes \cite{gene}. In other tasks, such as image denoising \cite{image}, the coefficients $\myvec{\theta}$ are of interest; based on the recovered support these could be estimated using least squares.\\
	
	%\subsection{Existing results.}
	In the regime $N<D$, the support recovery problem (illustrated in Fig. \ref{example}) is under-determined: we have fewer equations $N$ than variables $D$, and the observations are contaminated by additive noise $\myvec{\omega}$. In this setting sparsity is a useful assumption and it would be natural to estimate $\myvec{\theta}^*$ by minimizing
	$$
	\|\myvec{y}-\myvec{X}\myvec{\theta}\|^2_2 \quad \text{s.t.} \quad \|\myvec{\theta}\|_0 \leq k.
	$$  Since optimizing over this equation is intractable; several authors have replaced the $\ell^0$ norm by the $\ell^1$, which induces sparsity and leads to the well known Least Absolute Shrinkage and Selection Operator (LASSO) \cite{lasso}. {\rev We note that the sparsity properties of $\ell^p$ norms, for $p>1$ was studied in \cite{shen2018least}}. The LASSO, typically formulated using a regularized version of the problem, enjoys efficient optimization schemes \cite{lasso2, lasso3}. \cite{martin} showed that exact support recovery using the LASSO can occur with probability one if  $N > 2k \log(D-k) $. Several iterative methods for support recovery have been proposed, including:  Iterative Support Detection (ISD) \cite{isd}, the iteratively reweighted least squares (IRLS) \cite{irls} and the iteratively reweighted $\ell^1$ minimization (IRL1) \cite{irl1}.  
	The problem has also been addressed using greedy methods such as Orthogonal Matching Pursuit (OMP) \cite{omp}, Random OMP \cite{elad2009plurality} and other extensions \cite{stomp,cosamp}, or non convex schemes such as Trimmed LASSO (TL) \cite{tl} or smoothly clipped absolute deviation (SCAD) \cite{scad}. {\rev Recently, in \cite{shen2019exact}, the authors proposed a constrained matching pursuit algorithm for support recovery.}
	The importance of the problem has made it quite impossible to give an accurate, complete summary of the literature: we refer to the surveys  \cite{arjoune2017compressive,bruckstein2009sparse,marques2018review,mousavi2019survey}.
	
	{ \rev In this study, we propose RAWLS (\textbf{R}andomly \textbf{A}ggregated un\textbf{W}eighted \textbf{L}east Squares \textbf{S}upport Recovery), a simple scheme for support recovery from noisy measurements
		$$ \myvec{y} = \myvec{X}\myvec{\theta}^* + \myvec{\omega},$$
		where $\myvec{\theta}^* \in \left\{-1,0,1\right\}^D$ is a sparse vector. The ternary model for $\myvec{\theta}^*$ is motivated by several applications such as: compressing neural networks \cite{yan2016ternary,di2020compressing} or representing biological signals \cite{alemdar2017ternary}. RAWLS relies on subsampling the full set of equation and performing least squares on each subset. After averaging over the different solutions, we estimate the support using the most significant coefficients of the least squares solution. We prove a bound on the approximation of $\myvec{\theta}^*$ based on the proposed procedure. Finally, we demonstrate the applicability of RAWLS to the task of support recovery using different sparsity and noise levels. We observe, that our method outperforms several leading baselines in the low information regime.
		
		The paper is structured as follows. In Section \ref{sec:idea} we describe and motivate the proposed idea and our main results. Then, in Section \ref{sec:main} we provide a full description of RAWLS and demonstrate it efficacy using several examples. Finally, in Section \ref{sec:result} we prove our main result.}
	\section{The Idea and the Main Result}
	\label{sec:idea}
	\subsection{The Idea.}
	Our idea is quite simple: to estimate $\myvec{\theta}^*$, we will use least squares. This naive approach is a bad idea since 
	$$ \myvec{\widehat{\theta} }= \arg\min_{\myvec{\theta} \in \mathbb{R}^D} \| \myvec{X} \myvec{\theta} - \myvec{y}\|^2_2$$
	tends to require a fairly large number of queries $N$ to recover $\myvec{\theta}^*$ stably. The proposed scheme is based on the following observation: instead of running
	least squares on the full set of equations, we can use only a random subset of the equations. The underlying idea behind RAWLS 
	(\textbf{R}andomly \textbf{A}ggregated Un\textbf{w}eighted \textbf{L}east Squares \textbf{S}upport Recovery)\footnote{
		`The natural distribution is neither just nor unjust; nor is it unjust that persons are born into society at some particular position. These are simply natural facts. What is just and unjust is the way that institutions deal with these facts.' (John Rawls, 'A Theory of Justice' \cite{Rawls}).}
	is that none of the equations are distinguished: taking merely a subset of them amounts to a loss of information but provides a particularly unique point of view. However, since no particular subset of the equations is distinguished over any other subset, we average over a number of randomly selected subsets. Our analysis shows that this is indeed advantageous: while applying least squares using fewer equations leads to errors from the lack of information, these errors cancel (to some degree) when averaged.
	More precisely, let $A \subset \left\{1, \dots, N\right\}$,
	we define $\myvec{X}_A$ to be the restriction of $\myvec{X}$ onto the rows whose index is in the set $A$ and likewise for $\myvec{y}_A$. We then find
	\begin{equation}\label{eq:estimate}
		\myvec{\widehat{\theta}}_A = \arg\min_{\myvec{\theta} \in \mathbb{R}^D} \| \myvec{X}_A \myvec{\theta} - \myvec{y}_A\|^2_2.
	\end{equation} 
	We average this result over many subsets $(A_i)_{i=1}^{m}$ which we assume, for some fixed $n < \min(N,D)$, to be taken uniformly at random from all $n-$element subsets of $\left\{1, 2, \dots, N\right\}$ and use this as our estimate for a rescaling of $\myvec{\theta}^*$.  We hope that
	$$  \frac{1}{m} \sum_{i=1}^{m} \myvec{\widehat{\theta}}_{A_i} \sim  \frac{n}{D} \myvec{\theta}^*.$$

	\begin{figure}[h!] 
		\begin{center}
			\includegraphics[width=0.6\textwidth]{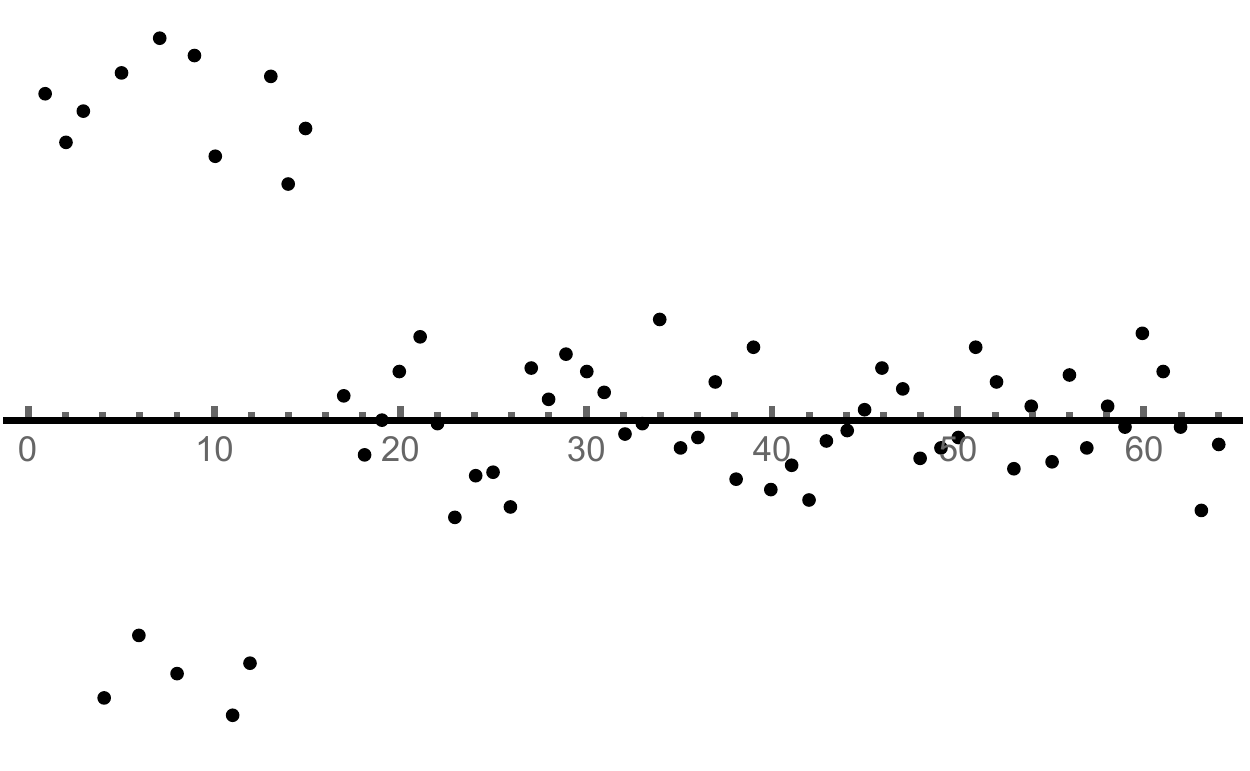}
			\caption{The reconstructed vector is much larger on the support of $\theta$ than off the support of $\theta$ and correctly identifies its sign.}
			\label{rec1}
		\end{center}
	\end{figure}

	An example (see Fig. \ref{rec1}) is as follows: let us define $\myvec{\theta} \in \mathbb{R}^{64}$ by setting the first $k=16$ entries to be $\pm 1$ (randomly) and the rest to be 0. We take a random Gaussian matrix $\myvec{X} \in \mathbb{R}^{64 \times 80}$, take subsets of size $n=58$ equations and average the least-square recovery over $m=100$ random choices of these $58$ equations.
	We observe that the reconstructed vector is much larger on the actual support than it is off the support; moreover, it correctly identifies the sign of the entry of $\myvec{\theta}$.

	\begin{figure}[h!] 
		\begin{center}
			\includegraphics[width=0.6\textwidth]{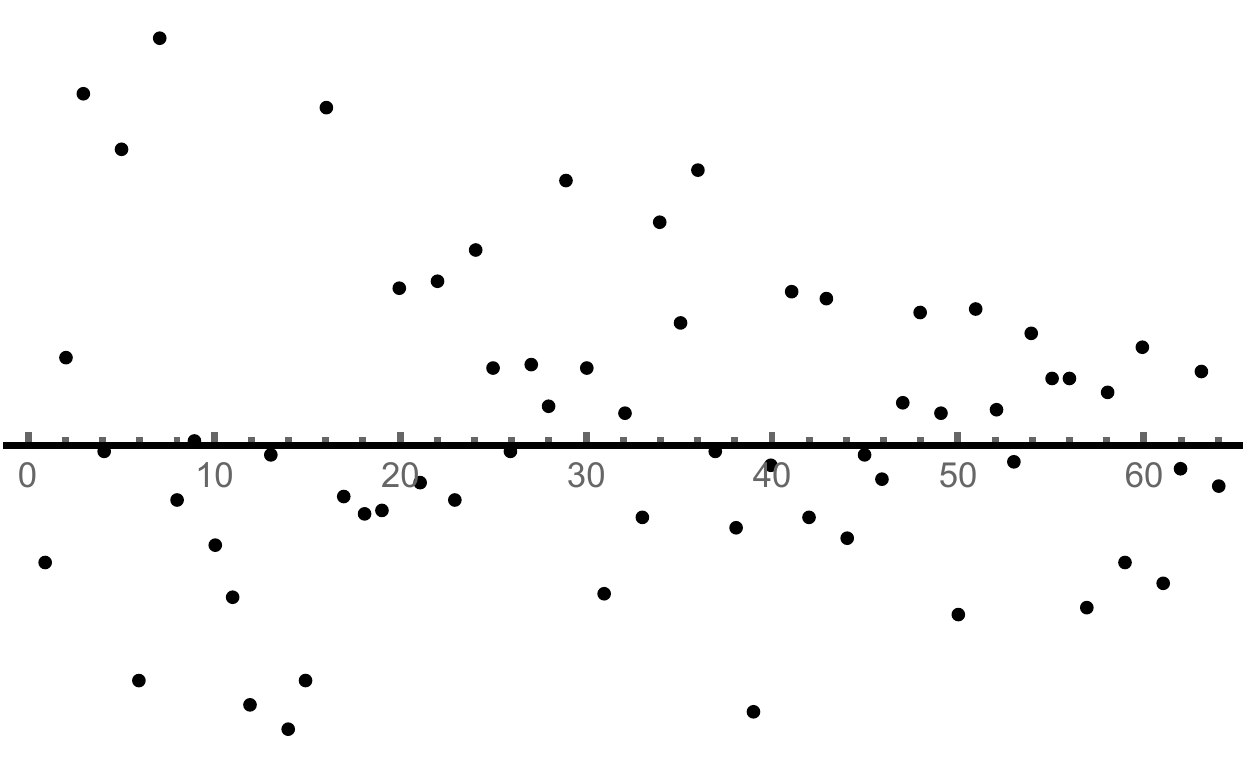}
			\caption{Reconstructing a noisy vector in $\mathbb{R}^D$, $D=64$ (supported on the first 16 coordinates) using $N=30$ equations (projected on $n=18-$dimensional subsets) with RAWLS.}
			\label{rec2}
		\end{center}
		
	\end{figure}

	Once we go down to a smaller number of equations $N$, something remarkable happens. For simplicity of exposition, we consider the same problem as above (reconstruction of a vector in $D=64$ dimensions) except now we only observe $N=30$ equations and we average over random subsets of these equations of size $n=18$. We emphasize that this quite the extreme setting; we are operating with very little information. This is reflected in the reconstructed vector (see Fig. \ref{rec2}): it is certainly not the case that the largest $k=16$ entries (by absolute value) correspond to the support of $\theta$. \textit{However}, what we observe in this setting is the largest entry is indeed located on the support of $\theta$: for this particular choice of parameters $(D,N,n) = (64,30, 18)$, this happens in $\sim90\%$ of all cases. This motivated our RAWLS-based peeling algorithm discussed in \S 3, where we iteratively remove the coordinate corresponding to the largest reconstructed vector. We note that correctly identifying the first coordinate is the most difficult task; after that we have reduced the problem by decreasing the size of the support, one less dimension $D \rightarrow D-1$, and the same number of equations $N$. This is an easier problem.

	\subsection{The Result} 
	We can show that this yields provably good results. Before formally stating the result, we will quickly outline its meaning. Instead of trying to recover the vector $\myvec{\theta}^*$, we will try to reconstruct its rescaled version $(n/D)\myvec{\theta}^*$ via an 
	$$ \mbox{average over random projections} \qquad  \frac{1}{m} \sum_{i=1}^{m} \pi_{A_i} \myvec{\theta}^*,$$
	where $\pi_{A_i}$ denotes the projection onto the subspace $A_i$ and the $A_i$ are, by an abuse of notation, subspaces of size $n$ chosen uniformly at random (subspaces spanned by the rows of $\myvec{X}$ indexed by $A_i$). However, we do not have access to $\myvec{\theta}^* \in \mathbb{R}^D$, we only have access to $\myvec{y} = \myvec{X} \myvec{\theta}^* + \myvec{\omega}$. Instead of taking a least squares projection of $\myvec{y}$, we will use the least squares projections of $\myvec{y}_{A_i}$ for random subsets of the equations in the hope that this approximately recovers $\myvec{\theta}^*$
	$$   \frac{1}{m} \sum_{i=1}^{m} \myvec{\widehat{\theta}}_{A_i} \sim  \frac{n}{D} \myvec{\theta}^*.$$

	\begin{thm} Let $\myvec{\theta}^* \in \mathbb{R}^D$ be an arbitrary vector. Then, by projecting onto subsets of $n < 0.9 \cdot D$ equations of the $N$ equations given by $\myvec{y} = \myvec{X}\myvec{\theta}^* + \myvec{\omega}$, we have
		$$ \mathbb{E}_{\textnormal{X},{\omega}}~ \left\|  \frac{1}{m} \sum_{i=1}^{m} \pi_{A_i} \myvec{\theta}^* - \frac{1}{m} \sum_{i=1}^{m}\myvec{\widehat{\theta}}_{A_i} \right\|_{\ell^2} \lesssim  \frac{n}{\sqrt{N} \sqrt{D-2}} + \frac{n}{D}.$$
	\end{thm}
	Several remarks are in order.
	\begin{enumerate}
		\item The statement is \textit{independent} of $\myvec{\theta}^*$. In particular, there is no underlying assumption about the structure of $\myvec{\theta}^*$ (and $\myvec{\theta}^*$ need not be sparse). We also observe that the size of $\myvec{\theta}^*$ does not appear on the right-hand side. This respects the problem setup where instead of $\myvec{X}\myvec{\theta}^*$ we are given the (additive) noisy version $\myvec{X} \myvec{\theta}^* + \myvec{\omega}$, where $\myvec{\omega}$ is a standard Gaussian vector ${\omega}_i \sim \mathcal{N}(0,1)$. 
		\item In the setting $n \leq N \ll D$ our analysis is accurate down to constants, both quantities on the right-hand side are asymptotically correct (in the sense of having the correct constant, 1, in front) if the scales separate more and more (see Figs. \ref{err1}, \ref{err2} and the Remark in \S 4.2).
		\item The randomness in the choice of the $A_i$ is not at all necessary. In fact, the proof suggests that one could simply pick completely deterministic subsets of the $N$ equations as long as none of the individual dimensions are featured too prominently and all are represented roughly an equal number of times in the projections. This is also substantiated by numerical evidence. This poses the question of whether there are `good' deterministic choices of subsets or whether there is a natural weight one could assign to the outcome resulting from each subset of equations (some `measure of reliability').
		\item The result suggests that using a smaller $n$ leads to a smaller error. However, it also leads to a smaller projection. We can compensate for that by inserting the appropriate scaling in our result from which we obtain
		$$ \mathbb{E}_{\textnormal{X},{\omega}}~ \left\|  \underbrace{ \frac{1}{m} \frac{D}{n} \sum_{i=1}^{m} \pi_{A_i} \myvec{\theta}^* }_{\approx \myvec{\theta}^*}- \frac{1}{m} \frac{D}{n} \sum_{i=1}^{m}\myvec{\widehat{\theta}}_{A_i} \right\|_{\ell^2} \lesssim  \frac{\sqrt{D}}{\sqrt{N}} + 1.$$
		This shows that there is some flexibility in the choice of $n$. In practice we have found that $n = 0.6\min(N,D)$ seems to be particularly suited (though not very different from, say, $n=0.5\min(N,D)$). The precise role of $n$ could be an interesting object for further study.
	\end{enumerate}
	
	We conclude by showing Theorem 1 in a simple example. As mentioned above, the terms in the upper bound (without the implicit constant and constant 1 instead) correspond to the sharp asymptotic limiting case where $n \leq N \ll D$. We show the case where $D=10000$, $N=100$ and $1 \leq n \leq 100$. For each value of $n$, we sample over $m=20$ random subsets of size $n$ of the $N$ equations.
	As for the vector $\theta$, it does not actually play a role, we chose it to be a Gaussian vector in $\mathbb{R}^D$. 
	We observe that the prediction is quite accurate (and the proof explains why this would be the case -- various quantities start concentrating tightly around their expectation).

	\begin{figure}[h!]
		\begin{center}
			\includegraphics[width=0.6\textwidth]{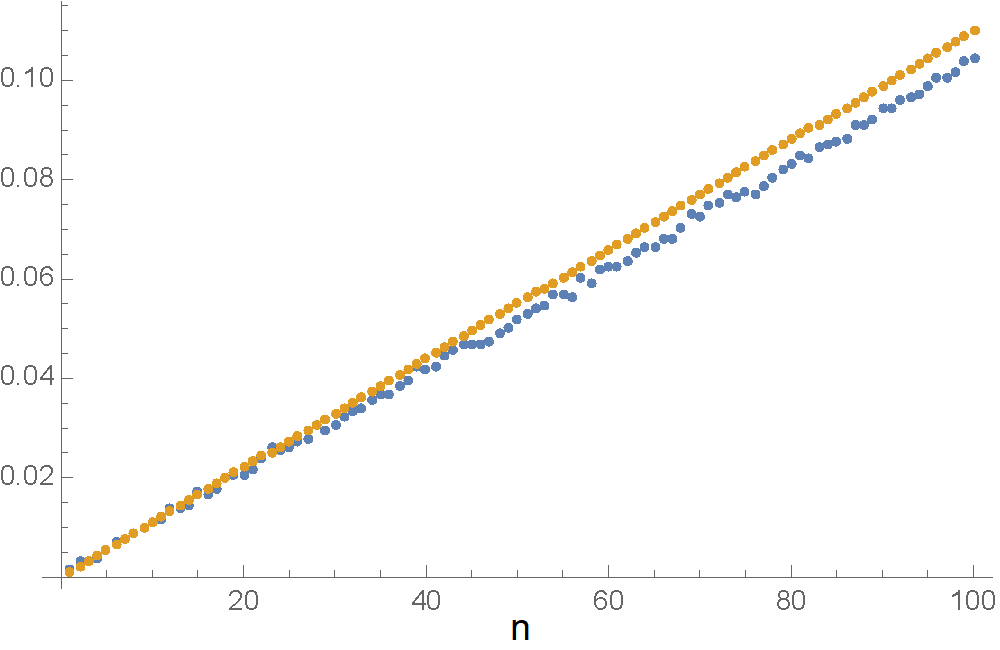}
			\caption{The error bounds in Theorem 1 (orange; ignoring the implicit constant) compared to the actual error (blue) for $D=10000$, $N=100$ and $1 \leq n \leq 100$.}
			\label{err1}
		\end{center}
		
	\end{figure}

	We also quickly illustrate that the restriction $n < 0.9 \cdot D$ is not just an artifact of the proof but, in fact, necessary (this also explains why RAWLS is better at recovering $\theta$ than an application of least squares to the full set of equations). We consider $\theta$ to be a unit vector (obtained from normalizing an instance of a Gaussian vector) in $D=200$ dimensions. We take $N=200$ equations and see what happens for $1 \leq n \leq 195$ (see Figure \ref{err2}).
	What we observe is that the theoretical error bound (with the implicit constant assumed to be 1) nicely dominates the error until $n$ starts getting very close to $D$ (we plot it for $1 \leq n \leq 195 < 200 = D$). We see that the error starts exceeding the size of the vector by many orders of magnitude. The proof will explain this as a degeneracy of the smallest singular value of a rectangular Gaussian matrix which becomes approximately square.

	\begin{figure}[h!] 
		\begin{center}
			\includegraphics[width=0.55\textwidth]{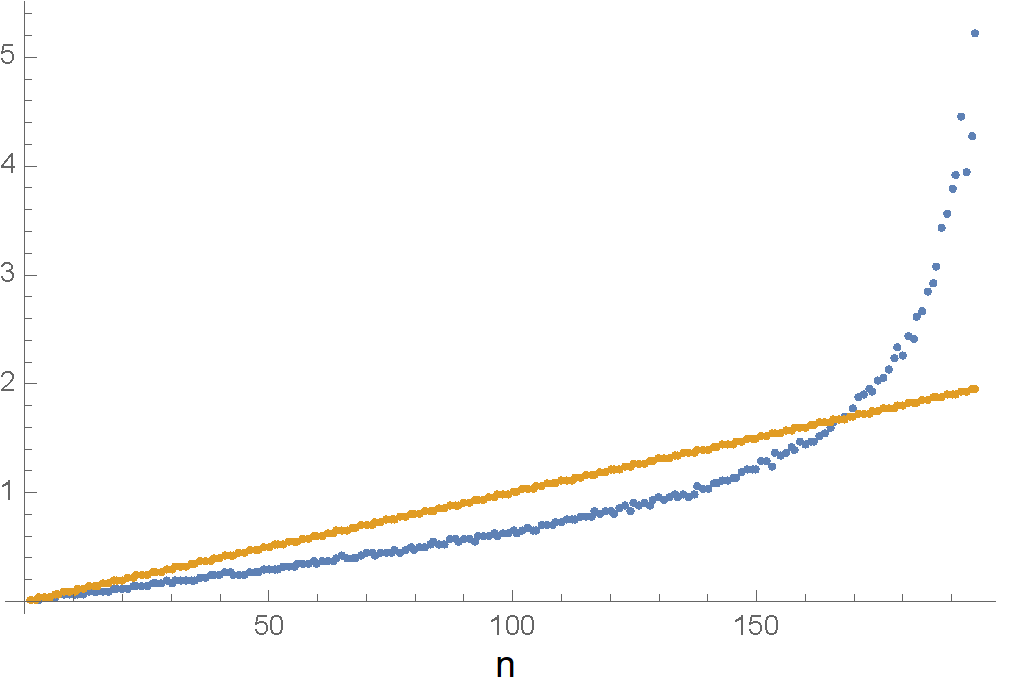}
			\caption{The error bounds in Theorem 1 (orange; ignoring the implicit constant) compared to the actual error (blue) for $D=200 =N$ and $1 \leq n \leq 195$.}
			\label{err2}
		\end{center}
	\end{figure}

	The error observed around $n=195 \sim N=D$ demonstrates why least squares using the full set of equations does not work; we obtain similar results also for $N \ll D$, the averaging has a natural stabilizing effect. We refer to the Remark in \S 4.2. for a prediction for what one would expect the error to look like when, say, $n = 0.99 D$.
	
	\subsection{Open Problems.}
	Theorem 1 raises a lot of open questions. 
	Is there a particularly natural choice of subspaces on which to project? We are investigating the case of random projections but the proof does not seem to require this; are there natural `adapted' subspaces that one can derive from a given matrix $X$? \\
	We conclude with a particularly interesting question. We recall that the random projections reduce the size of the resulting vector. We can compensate for that by inserting the appropriate scaling in our result from which we obtain
	$$ \mathbb{E}_{\textnormal{X},{\omega}}~ \left\|  \underbrace{ \frac{1}{m} \frac{D}{n} \sum_{i=1}^{m} \pi_{A_i} \theta }_{\approx \theta}- \frac{1}{m} \frac{D}{n} \sum_{i=1}^{m} x_{A_i}^* \right\|_{\ell^2} \lesssim  \frac{\sqrt{D}}{\sqrt{N}} + 1.$$
	In the case where we assume $\theta \in \left\{-1, 0, 1\right\}^D$, we want to make sure that we are properly able to distinguish two different vectors of that type and this can then be seen to require $N \sim D$ (not entirely surprising, we are not making any assumptions on the sparsity of $\theta$). However, a more refined approach is conceivable: ultimately, we are using the entries of our approximating vector to derive statements about the support. As such, the $\ell^2$ is perhaps not the only interesting quantity and estimates on $\ell^{\infty}$ would be quite desirable. In particular, what we observe in practice (and what motivated the peeling algorithm) is that very large entries (either very large or very small) in the recovered approximation is a good indicator for $\theta$ having support in that coordinate. This simple observations forms the basis of the algorithm discussed in \S 3. It would be interesting to have results in that direction.\\
	
	We also emphasize that the idea underlying RAWLS might have many other applications: it is ultimately an $\ell^2-$based concept and as such many natural variations seem conceivable. One such applications, a nonlinear variant that is shown to work particularly well in the support recovery problem, is discussed in the next section. A second question, outside the scope of this paper, is whether other methods used for support recovery could conceivably be merged with our philosophy: running it on random subsets of the equations and hoping that the averaging effects compensates for the loss of information.

	\section{Support Recovery with RAWLS}  \label{sec:main}
	
	\subsection{ The Idea.}
	\textit{If} it is indeed the case that
	$$  \frac{1}{m} \sum_{i=1}^{m} \pi_{A_i} \myvec{\theta}^* \approx \frac{n}{D}\myvec{\theta}^* + \mbox{some error}$$
	and \textit{if} the error is nicely random (as one usually expects in these cases), then the largest (or smallest) entries of the vector should be contained in the support of $\myvec{\theta}^*$. In a more elementary formulation, if we are given $\myvec{v} \in \left\{-1, 0,1\right\}^D$ (such that $\|\myvec{v}\|_{0}$ is not too small compared to $D$) and add a random Gaussian vector $\myvec{g}$ to it, then the largest (absolute) entry of $\myvec{v}+\myvec{g}$ will be attained (with high probability) on the support of $\myvec{v}$. This is a simple consequence of the rapid decay of the Gaussian, and motivates the Algorithm proposed in the following subsection.
	
	\subsection{Peeling with RAWLS}
	\begin{enumerate} 
		\item Compute the approximation 
		$$ \myvec{\widehat{\theta}} = \frac{1}{m} \sum_{i=1}^{m}\myvec{\widehat{\theta}}_{A_i},$$
		where $\myvec{\hat{\theta}}$ is estimated based on Eq. \ref{eq:estimate}.
		\item Find the element of largest absolute value of $\myvec{\widehat{\theta}}$. If this element $\widehat{\theta}_{\ell}$ is positive, then we assume that ${\theta}^*_{\ell}=1$; if negative, we assume ${\theta}^*_{\ell}=-1$.
		\item Remove the corresponding column from the matrix $\myvec{X}$ and update the right-hand side $\myvec{y}$ by subtracting the $\myvec{X}$ projected onto the estimated coordinate $\widehat{\theta}_{\ell}$.
		\item Return to (1) until $k$ non-zero entries of $\widehat{\myvec{\theta}}$ are estimated.
	\end{enumerate}

	This algorithm is thus a fairly simple greedy algorithm that identifies likely candidates for the support of $\myvec{\theta}^*$ by looking for particularly large entries in the RAWLS-reconstruction of $\myvec{\theta}^*$. We emphasize that for this type of iterative algorithm, each step is more difficult than the next one: having found a correct entry, the problem is reduced to a simpler problem $D \rightarrow D-1$ while maintaining the same amount of information $N \rightarrow N$. We point out that the method, just as other methods, should also be suitable for \textit{partial} recovery: finding a set of $k$ entries that has a large overlap with the ground truth, we do not pursue this here. We do not have any theoretical guarantees for the success rate of the peeling algorithm at this point and consider this to be an interesting problem. Perhaps the most interesting question at this stage is whether there are other implementations of these underlying ideas that can yield even better results.

	\subsection{Numerical Performance}
	In this section we support the effectiveness of RAWLS using numerical simulations. We focus on the task of exact support recovery using a random Gaussian design matrix $\myvec{X}$ with values drawn independently from $N(0,1)$ and random additive Gaussian noise $\myvec{\omega}$ with values drawn independently from $N(0,\sigma^2)$. As baselines, we compare the method to LASSO \cite{lasso}, IRL1 \cite{irl1}, TL \cite{tl}, OMP \cite{omp}, RandOMP \cite{elad2009plurality} and STG \cite{stg}. 
	To evaluate the probability of exact support recovery we run each method $100$ times and count the portion of successful estimations. A successful estimation of the support is counted if $S(\myvec{\theta})=S(\widehat{\myvec{\theta}})$, where $S(\myvec{\theta}):= \{i \in 1,...,D | \myvec{\theta}_i \neq 0 \}$. To improve the stability of LASSO, after each run we select the top $k$ coefficients of $\myvec{\theta}$ as the estimated support.  
	
	First, in Figure. \ref{fig:phase_tran} we present the probability of successful support recovery using $D=64$ variables, a fixed sparsity of $k=10$ and different number of measurements $N$. This example demonstrates that RAWLS can successfully recover the unknown support with fewer measurement compared with OMP,RandOMP and the LASSO.

	Next, we demonstrate how the sparsity level $k$ affects the success of RAWLS in recovering the support of $\myvec{\theta}^*$. We use $100$ simulations with design matrix $\myvec{X}$ and Gaussian noise $\myvec{\omega}$ defined as in the previous example and evaluate the performance of RAWLS for sparsity levels $k$ in $\{2,4,...,34\}$. In Figure \ref{fig:k_eval} demonstrate that RAWLS success rate is comparable to IRL1, STG and TL for large values of $k$. 
	
	{\rev We further evaluate the performance of RAWLS for other type of measurement matrices. Specifically, we generate a Bernoulli design matrix $\myvec{X}$ taking values $\{-1,1 \}$ with equal probability. We use additive Gaussian noise with zero mean and standard deviation of $\sigma=0.5$. Here, we also compare RAWLS to stochastic resonance OMP (SR\_OMP) \cite{simon2019mmse}. In Fig. \ref{fig:bin} we present the probability of support recovery using RAWLS and several other baselines. In this example, the sparsity level is $k=10$ with dimension $D=64$ and the results are based on $100$ simulations.
		
		Finally, we evaluate the performance of RAWLS in the regime of low information. Specifically, we use a vector with a sparsity level $k=30$, with $D=64$ variables, and focus on the regime of $30 \leq N \leq 90 $. Here, we restrict our comparison to the leading baselines, namely to IRL1 and TL. As observed in Fig. \ref{fig:hard}, RAWLS outperforms IRL1 and TL, in the regime of low information. Precisely, for $N\leq 55$, RALWS recovers the support with a higher probability compared with the competing methods.
	}

	\begin{figure}[tb!]
		
		\centering
		\includegraphics[width=0.85\textwidth]{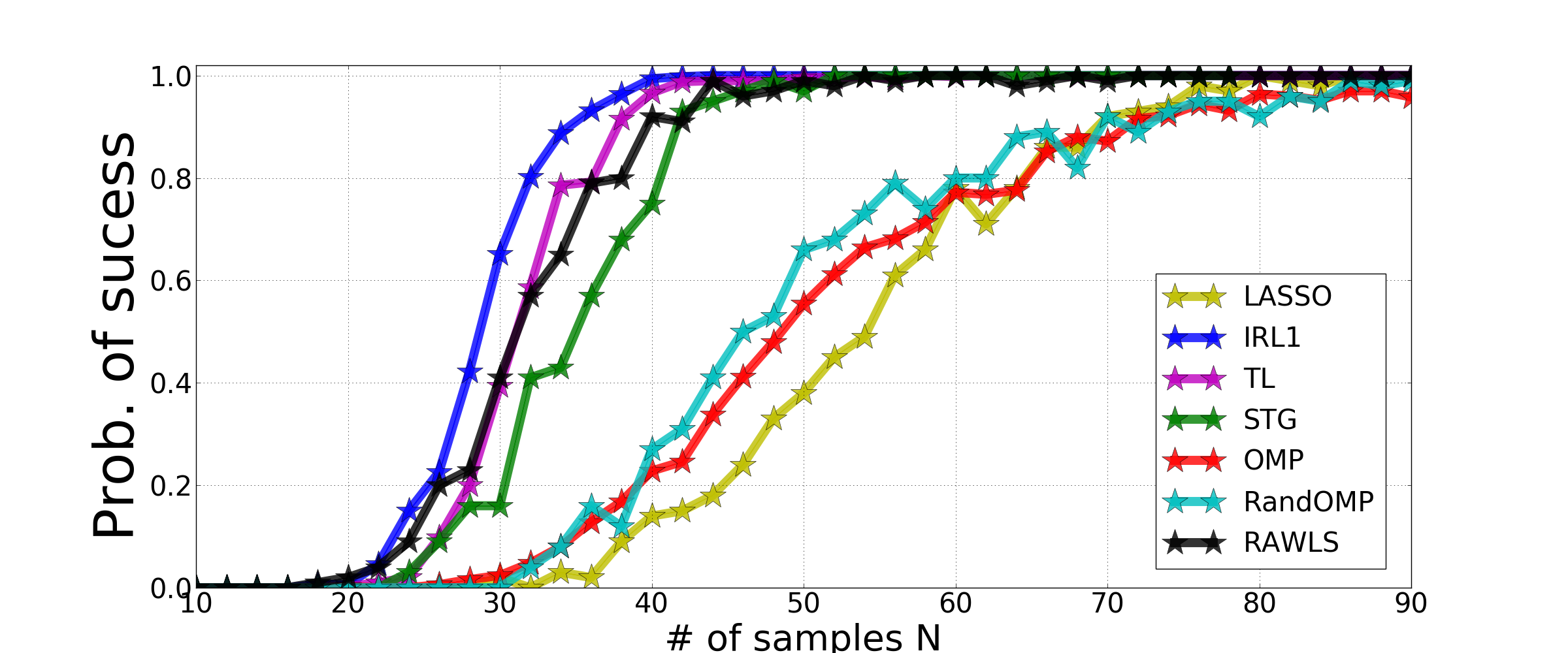}
		\includegraphics[width=0.85\textwidth]{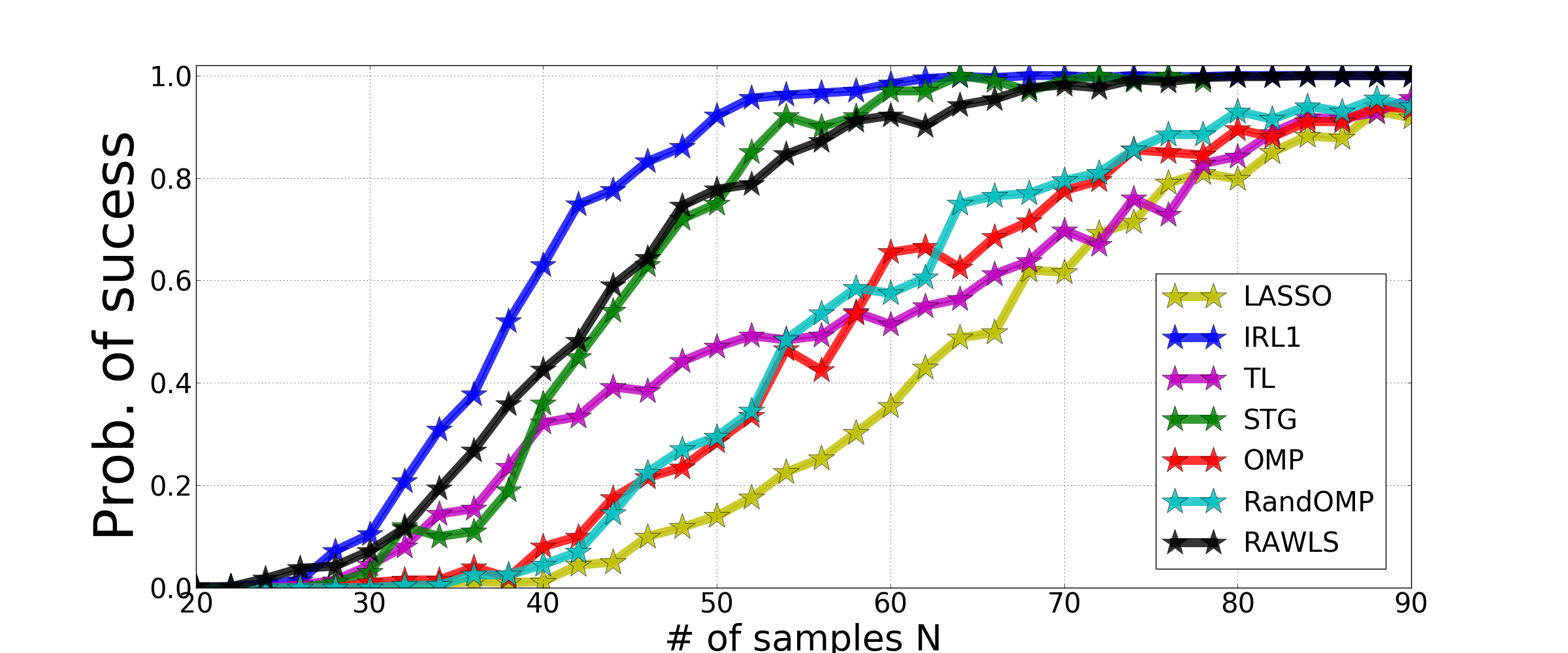}
		
		\caption{Numerical evaluation for the probability of exact support recovery vs. number of measurements $N$. We compare Peeling with RAWLS to several baselines for:  $\sigma=0.5$ (top panel) and $\sigma=1$ (bottom panel).}
		
		\label{fig:phase_tran}
	\end{figure}

	\begin{figure}[tb!]
		
		\centering
		\includegraphics[width=0.85\textwidth]{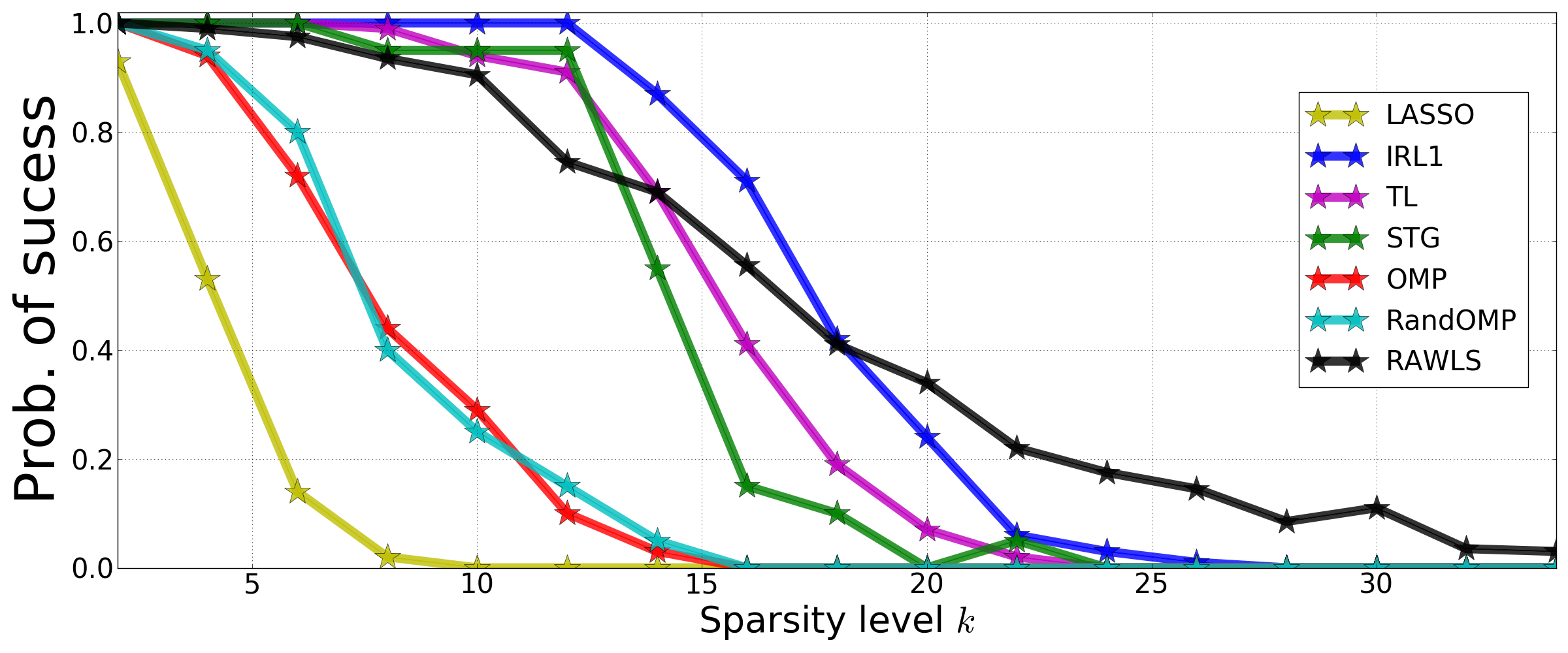}
		\includegraphics[width=0.85\textwidth]{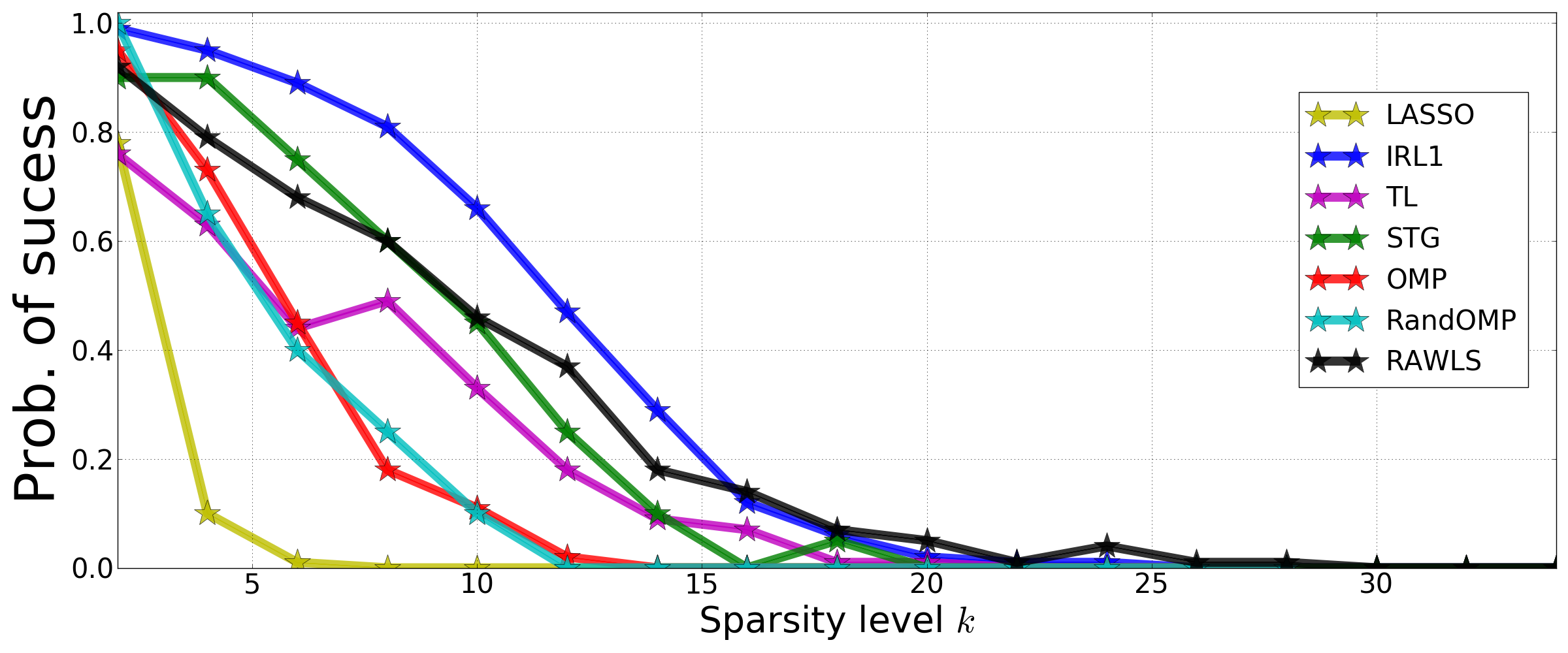}
		
		\caption{Numerical evaluation for the probability of exact support recovery vs. sparsity level $k$. Here the number of variables and measurements are fixed, specifically $D=64, \text{ and }N=40$. We compare Peeling with RAWLS to several baselines for:  $\sigma=0.5$ (top panel) and $\sigma=1$ (bottom panel).}
		
		\label{fig:k_eval}
	\end{figure}

	\begin{figure}[tb!]
		
		\centering
		\includegraphics[width=0.85\textwidth]{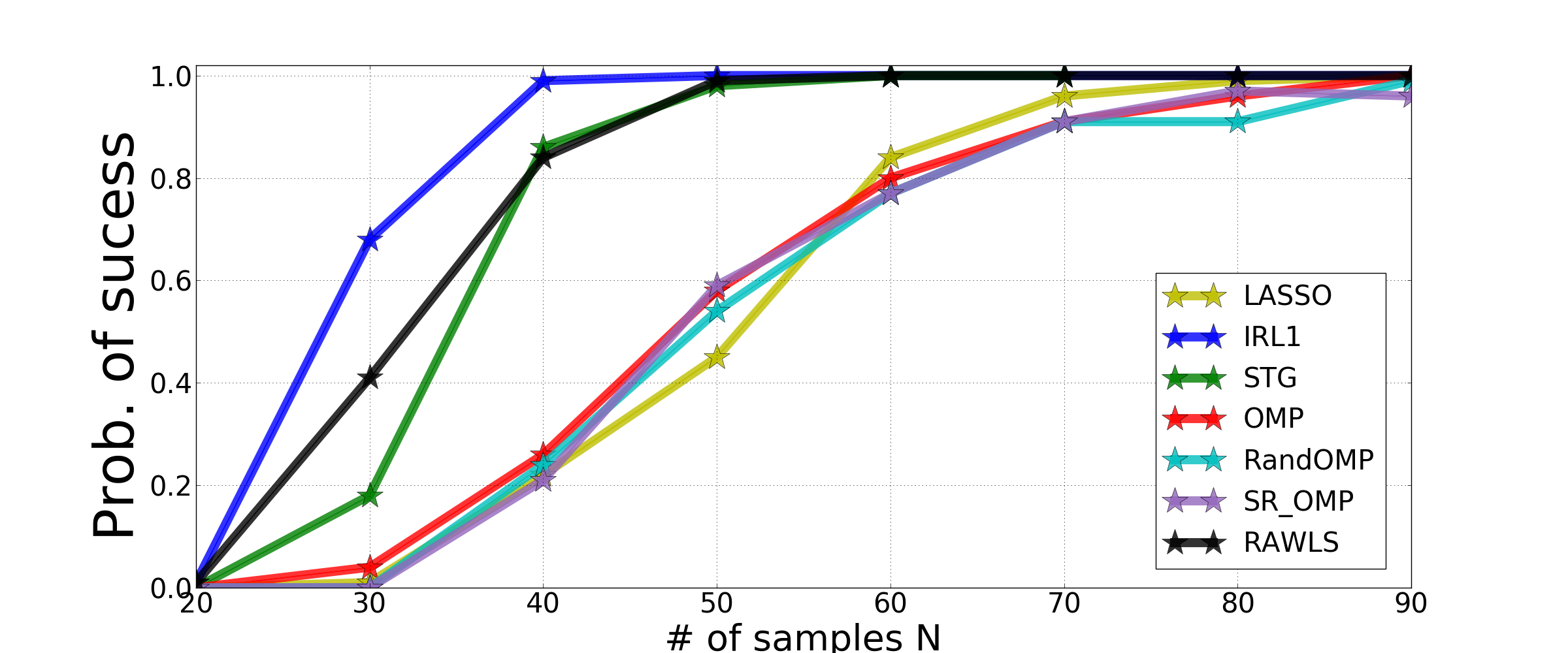}
		\caption{Numerical evaluation for the probability of exact support recovery vs. number of measurements $N$. Here the number of variables and sparsity are fixed, specifically $D=64, \text{ and }k=10$. We compare Peeling with RAWLS to several baselines for a binary design matrix with values drawn from a fair Bernoulli distribution. The additive noise is Gaussian with zero mean and standard deviation $\sigma=0.5$}
		
		\label{fig:bin}
	\end{figure}

	\begin{figure}[tb!]
		
		\centering
		\includegraphics[width=0.85\textwidth]{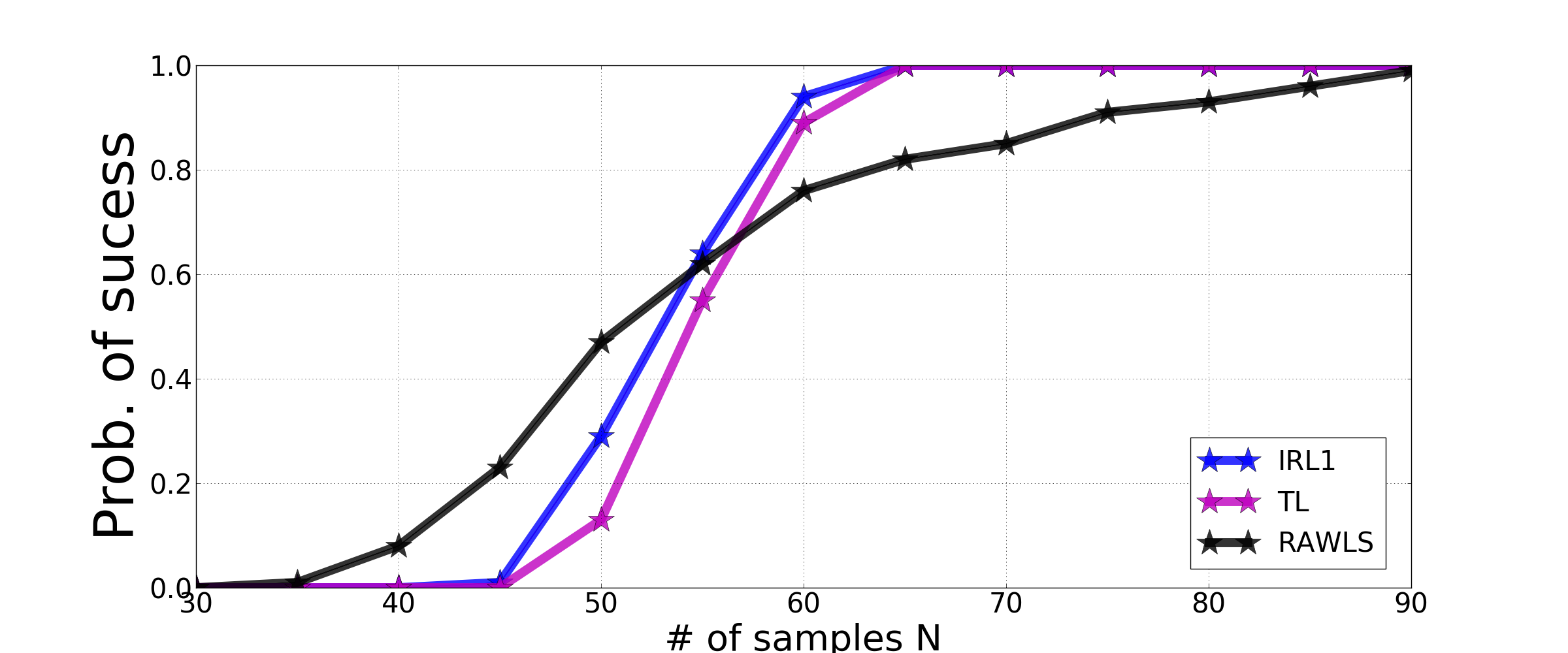}
		\caption{Numerical evaluation for the probability of exact support recovery vs. number of measurements $N$. Here the number of variables and sparsity are fixed, specifically $D=64, \text{ and }k=30$. The additive noise is Gaussian with zero mean and standard deviation $\sigma=0.5$. RAWLS outperforms state of the art method in the low information regime.}
		
		\label{fig:hard}
	\end{figure}

	\section{Proof of the Theorem} \label{sec:result}
	\subsection{Setup.} 
	Let $\myvec{\theta}^* \in \left\{-1,0,1\right\}^D$ be a sparse vector with support $\|\myvec{\theta}^*\|_0 = k$ and let $\myvec{X} \in \mathbb{R}^{N \times D}$ be a design matrix
	all of whose entries are i.i.d. random variables drawn from $\mathcal{N}(0,1)$. We will also use the notation $g=(\myvec{g}_i)_{i=1}^{N}$ to denote
	the Gaussian vectors in $\mathbb{R}^D$ dimensions that are forming the rows. We are given
	$$ \myvec{y} = \myvec{X} \myvec{\theta}^* + \myvec{\omega},$$
	where each entry of $\myvec{\omega}$ is i.i.d. normally distributed $\myvec{\omega}_i \sim \mathcal{N}(0,1)$. We try to understand how our algorithm performs
	on this data. Let $A \subset \left\{1, \dots, N\right\}$ be a random subset of size $|A| = n$. We are trying to understand the least squares
	solution in Eq. \ref{eq:estimate}
	where $\myvec{X}_A$ denotes the restrictions onto the rows of $\myvec{X}$ indexed by $A$ and likewise for $\myvec{y}_{A}$. If $n \leq D$, then the system has more variables than equations and always has a solution: we are interested in the solution with the smallest $\ell^2-$norm and will denote it by $\myvec{\widehat{\theta}}_{A}$.

	\subsection{A Single Projection.}
	The purpose of this statement is to provide the analysis of a single projection onto a random subspace spanned by a random subset of the rows. The main insight is that this projection can be approximately deconstructed into the projection of the ground truth, a highly structured Gaussian error on top of that and a relatively small error term.
	\begin{lemma} Let $\myvec{\theta}^* \in \mathbb{R}^D$ be fixed, let $\myvec{X} \in \mathbb{R}^{N \times D}$ be a random Gaussian matrix and let $A \subset \left\{1, 2, \dots, N\right\}$ be a randomly chosen subset of size $|A|=n < 0.9 \cdot D$. Then the orthogonal (noisy) projection of $\myvec{\theta}^*$ onto the subspace spanned by the rows indexed by $A$ (given by $\myvec{y}=\myvec{X}\myvec{\theta}^* + \myvec{\omega}$) satisfies
		$$ \widehat{\myvec{\theta}}_A = \pi_{A}\myvec{\theta}^* + \left(\sum_{a \in A}^{}{ \frac{\myvec{g}_a}{\|\myvec{g}_a\|^2} \myvec{\omega}_a}\right) + \myvec{e},$$
		{\rev where $g_a$ is the $a$-th row of the matrix $\myvec{X}$}, and $\myvec{e}$ satisfies, with high likelihood,
		$ \mathbb{E}_{\textnormal{X},{\omega}}~ \|\myvec{e}\| \lesssim \frac{n}{D}.$
	\end{lemma}
	The purpose of this Lemma is to show that the (noisy) projection of $\myvec{\theta}^*$ onto a random subspace (this is one interpretation of $\myvec{y} = \myvec{X} \myvec{\theta}^* + \myvec{\omega}$) leads to substantial distortions; however, these distortions are \textit{not} arbitrary and follow a fairly regular pattern up to a small error.
	The second term is not necessarily that small; however, its form will allow us to show that averaging it over multiple subspaces will further decrease the size. We emphasize that in the case $n \ll D$ our estimate is sharp and we expect $\|\myvec{e}\| \sim n/D$ with tight concentration and a small error (this could be made precise when $D/n$ becomes large).
	
	The proof makes use of the following basic fact in linear algebra that we recall for the convenience of the reader: let $(\myvec{g}_a)_{a=1}^{n}$ be $n$ vectors in $\mathbb{R}^D$ with $D > n$ and let $\myvec{v} \in \spasn\left\{\myvec{g}_1, \dots, \myvec{g}_n\right\}$. Then
	$$\sigma_{\min}(\myvec{G})^2 \|\myvec{v}\|^2 \leq \sum_{a=1}^{n}{ \left| \left\langle \myvec{g}_a, \myvec{v} \right\rangle \right|^2} \leq \sigma_{\max}(\myvec{G})^2 \|\myvec{v}\|^2,$$
	where $\sigma$ denotes the singular values of the matrix $\myvec{G}$ obtained by collecting $\{\myvec{g}_a\}^n_{a=1}$ as column vectors (or, alternatively, the largest and smallest eigenvectors of $\myvec{G}^T \myvec{G}$). This follows easily from observing that
	$$  \sum_{a=1}^{n}{ \left| \left\langle \myvec{g}_a, \myvec{v} \right\rangle \right|^2} = \|\myvec{G}^T \myvec{v}\|^2.$$
	This is well-known in frame theory: the frame constants for finite-dimensional problems are given by the singular values of the associated matrix.
	
	\begin{proof}[Proof of the Lemma] We will use $\widehat{\myvec{\theta}}_A$ to denote the $\ell^2-$smallest vector satisfying Eq. \ref{eq:estimate}.
		This solutions is contained in the vector space $ V = \spasn_{} \left\{\myvec{g}_a: a \in A \right\}$ (if $\myvec{\theta}$ had a component that was orthogonal to these rows, then it would not have any effect in the matrix multiplication $\myvec{X}_A \myvec{\theta}$ and removing that component would result in a smaller $\ell^2-$norm). Since the number of variables, $D$, is larger than the number of equations, $n$, and $\myvec{X}$ is Gaussian we know that the minimum is 0 with likelihood 1. Thus $\myvec{X}_A \widehat{\myvec{\theta}}_A= \myvec{y}_A = \myvec{X}_A \myvec{\theta}^* + \myvec{\omega}_A$. We will analyze this equation for a single row. For any $a \in A$,
		\begin{align*}
			\left\langle \myvec{g}_a, \widehat{\myvec{\theta}}_A \right\rangle = \left\langle \myvec{g}_a, \myvec{\theta}^* \right\rangle + \myvec{\omega}_a = \left\langle \myvec{g}_a, \myvec{\theta}^* + \frac{\myvec{\omega}_a}{\|\myvec{g}_a\|^2}\myvec{g}_a\right\rangle.
		\end{align*}
		We will use this equation for all $a \in A$. By definition of the orthogonal projection, we have, for all $a \in A$,
		$ \left\langle \myvec{g}_a, \myvec{\theta}^* \right\rangle = \left\langle \myvec{g}_a, \pi_A \myvec{\theta}^* \right\rangle,$
		and thus the identity
		\begin{equation}\label{eq:identity}
			\left\langle \myvec{g}_a, \widehat{\myvec{\theta}}_A \right\rangle =  \left\langle \myvec{g}_a, \pi_A \myvec{\theta}^* + \frac{\myvec{\omega}_a}{\|\myvec{g}_a\|^2}\myvec{g}_a\right\rangle.\end{equation}
		This is an interesting way of interpreting the introduction of additive noise: the error that we are given makes it seem as if the inner product was not with $\pi_A \myvec{\theta}^*$ but instead with $\pi_A \myvec{\theta}^*$ and a small additional multiple of $\myvec{g}_{a}$.
		In practice, if $n \ll D$, then the Gaussian vectors are ``almost'' orthogonal and ``almost'' form an orthogonal basis of the space that they span. This motivates the \textit{ansatz}
		$$ \widehat{\myvec{\theta}}_A = \pi_{A}\myvec{\theta}^* + \left(\sum_{a \in A}^{}{ \frac{\myvec{g}_a}{\|\myvec{g}_a\|^2} \myvec{\omega}_a }\right) + \myvec{e},$$
		where $\pi_A$ is the orthogonal projection onto the vector space
		$ V = \spasn_{} \left\{\myvec{g}_a: a \in A \right\}$
		and
		$\myvec{e} \in \mathbb{R}^D$ is an error term whose size we try to investigate. 
		We plug in our ansatz in to Eq. \ref{eq:identity} and obtain, for all $a \in A$,
		$$ \left\langle \myvec{g}_a, \myvec{e} \right\rangle = - \left\langle \myvec{g}_a,  \sum_{ a  \neq i \in A}^{}{ \frac{\myvec{g}_i}{\|\myvec{g}_i\|^2} \myvec{\omega}_i } \right\rangle.$$
		We emphasize that, since the $\myvec{g}_a$ span $V$ with probability 1, these $n$ equations uniquely identify $\myvec{e} \in V$ with probability 1.
		We first try to understand the quantity on the right-hand side. We have
		$$ \left\langle \myvec{g}_a,  \sum_{ a \neq i \in A}^{}{ \frac{\myvec{g}_i}{\|\myvec{g}_i\|^2} \myvec{\omega}_i } \right\rangle = \sum_{ a \neq i \in A}^{} \frac{\left\langle \myvec{g}_i, \myvec{g}_a \right\rangle}{\|\myvec{g}_i\|^2} \myvec{\omega}_i.$$
		The inner product of two random Gaussians is a random variable at scale $\left\langle \myvec{g}_i, \myvec{g}_a \right\rangle \sim \sqrt{D}$, the size of an individual Gaussian vector is at scale $\mathbb{E} \|\myvec{g}_i\|^2 \sim D + \mathcal{O}(\sqrt{D})$ with high likelihood. The $\myvec{\omega}_i \sim \mathcal{N}(0,1)$ have an additional randomization effect. The sum runs over $n-1$ elements. Altogether, we expect the quantity to be a random variable at scale
		$$ \left|  \sum_{ a \neq i \in A}^{} \frac{\left\langle \myvec{g}_i, \myvec{g}_a \right\rangle}{\|\myvec{g}_i\|^2} \myvec{\omega}_i \right|  \sim \frac{\sqrt{n}}{\sqrt{D}}.$$
		An explicit computation shows that
		\begin{align*}
			\mathbb{E} \left|  \sum_{ a \neq i \in A}^{} \frac{\left\langle \myvec{g}_i, \myvec{g}_a \right\rangle}{\|\myvec{g}_i\|^2} \myvec{\omega}_i \right|^2 = \mathbb{E} \sum_{ a \neq i \in A}^{} \frac{\left\langle \myvec{g}_i, \myvec{g}_a \right\rangle^2}{\|\myvec{g}_i\|^4} \myvec{\omega}_i^2 \\
			+ \mathbb{E} \sum_{ a \neq i_1 \neq i_2 \in A}^{} \frac{\left\langle \myvec{g}_{i_1}, \myvec{g}_a \right\rangle}{\|\myvec{g}_{i_1}\|^2}  \frac{\left\langle \myvec{g}_{i_2}, \myvec{g}_a \right\rangle}{\|\myvec{g}_{i_2}\|^2} \myvec{\omega}_{i_1} \myvec{\omega}_{i_2}.
		\end{align*}
		The second expectation is clearly 0 since $\myvec{\omega}_i \sim \mathcal{N}(0,1)$ and these are independent of each other.
		It remains to evaluate the first expectation. %Since $\mathbb{E} \|\myvec{g}_i\|^2 \sim D + \mathcal{O}(\sqrt{D})$ with exponentially decaying tail, we get 
		{\rev Since $\myvec{\omega}_i$ are independent of $\myvec{g}_i$, we get}
		$$  \mathbb{E}_{{\textnormal{X},{\omega}}} \sum_{ a \neq i \in A}^{} \frac{\left\langle \myvec{g}_i, \myvec{g}_a \right\rangle^2}{\|\myvec{g}_i\|^4} \myvec{\omega}_i^2 = 
		\mathbb{E}_{\textnormal{X}} \sum_{ a \neq i \in A}^{} \frac{\left\langle \myvec{g}_i, \myvec{g}_a \right\rangle^2}{\|\myvec{g}_i\|^4}.$$
		This sum can be decoupled into two parts
		$$  \mathbb{E}_{{\textnormal{X}}} \sum_{ a \neq i \in A}^{} \frac{\left\langle \myvec{g}_i, \myvec{g}_a \right\rangle^2}{\|\myvec{g}_i\|^4} =  \mathbb{E}_{{\textnormal{X}}} \sum_{ a \neq i \in A}^{} \left\langle \frac{\myvec{g}_i}{\|\myvec{g}_i\|}, \myvec{g}_a \right\rangle^2 \frac{1}{\|\myvec{g}_i\|^2}.$$
		We observe that $\myvec{g}_i/\|\myvec{g}_i\|$ is a random vector on the unit sphere (this follows from the rotational symmetry of Gaussian vectors); as such, it is completely independent of its length $\|\myvec{g}_i\|$ allowing us to treat both quantities as independent random variables. However, the first term is simply an inner product of a Gaussian vector against a unit length vector, thus
		$$\left\langle  \frac{\myvec{g}_i}{\|\myvec{g}_i\|}, \myvec{g}_a \right\rangle  \mbox{ is a Gaussian variable and  }  \mathbb{E} \left\langle  \frac{\myvec{g}_i}{\|\myvec{g}_i\|}, \myvec{g}_a \right\rangle^2 = 1.$$
		The remaining quantity is the mean of an inverse $\chi-$distribution which is $1/(D-2)$ for $D \geq 3$ and thus
		$$  \mathbb{E}_{X} \sum_{ a \neq i \in A}^{} \frac{\left\langle \myvec{g}_i, \myvec{g}_a \right\rangle^2}{\|\myvec{g}_i\|^4} =  \mathbb{E}_{X} \sum_{ a \neq i \in A}^{}  \frac{1}{\|\myvec{g}_i\|^2} = \frac{n-1}{D-2} \lesssim \frac{n}{D}.$$
		Summing up, we obtain
		$ \mathbb{E} \sum_{a \in A}  \left|\left\langle \myvec{g}_a, \myvec{e} \right\rangle \right|^2 \lesssim \frac{n^2}{D}.$
		However, since $e \in \spasn\left\{\myvec{g}_a: a \in A\right\}$, we have the basic inequality
		$$ \sigma_{\min}^2 \|\myvec{e}\|^2 \leq  \sum_{a \in A}  \left|\left\langle \myvec{g}_a, \myvec{e} \right\rangle \right|^2 \leq \sigma_{\max}^2 \|\myvec{e}\|^2.$$
		The smallest singular value of a random rectangular Gaussian matrix was determined by Silverstein \cite{silver} who showed
		that we can expect, in the limit, that
		$ \sigma_{\min} \sim  \sqrt{D} - \sqrt{n}.$
		Combining all these results shows that we expect, in the regime where $d$ has a bounded gap from $D$, say $n \leq 0.9 \cdot D$, that
		$$ \|\myvec{e}\| \lesssim_{} \frac{n}{D}.$$
	\end{proof}
	
	\textit{Remark.} We observe that the first part of the argument is fairly tight, in particular, we expect 
	$$\mathbb{E} \sum_{a \in A}  \left|\left\langle \myvec{g}_a, \myvec{e}\right\rangle \right|^2 \sim \frac{n^2}{D}$$
	with tight concentration. The second part of the argument is not precise down to constants but it becomes tight if we have $n \ll D$. 
	We observe that if $n \ll D$, then we actually have $\sigma_{\min} \sim \sigma_{\max}$ since the singular are expected to be in the interval $[\sqrt{D}-\sqrt{n}, \sqrt{D} + \sqrt{n}]$. Since all the estimates we carried out are actually quite tightly concentrated, we thus expect, with a fair degree of accuracy,
	$$ \|\myvec{e}\| \sim \frac{n}{D}.$$
	More precise, estimates are conceivable: if $\myvec{e}$ is uniformly distributed across all singular vectors, then we could hope that
	$$  \frac{1}{\|\myvec{e}\|^2} \sum_{a \in A}  \left|\left\langle \myvec{g}_a, \myvec{e} \right\rangle \right|^2 \sim Z^2,$$
	where $Z$ is the Marchenko-Pastur distribution modeling the singular values of the random matrix $X$. When $n \ll D$, then $Z \sim \sqrt{D} \pm \sqrt{n} \sim \sqrt{D}$ and we recover the usual estimate. As soon as $n$ starts approaching $D$, the distribution of $Z$ gets closer and closer to 0 and the inverse distribution $1/Z^2$ spreads over many scales. However, in principle, if $\myvec{e}$ is uniformly distributed over the singular vectors, then one could use this heuristic to predict the sharp constant to be expected when, for example $n = 0.99 \cdot D$. Basic numerics seems to indicate that this is a reasonable assumption.
	
	\subsection{Multiple Projections.} We now discuss the effect of averaging quantities like
	$$\sum_{a \in A}^{}{ \frac{\myvec{g}_a}{\|\myvec{g}_a\|^2} \myvec{\omega}_a}$$
	over multiple randomly chosen sets $A$. 
	\begin{lemma} Let $\myvec{X} \in \mathbb{R}^{N \times D}$ be a matrix with i.i.d. standard $\mathcal{N}(0,1)$ entries and let $\myvec{\omega} \in \mathbb{R}^N$ be a random vector all of whose entries are i.i.d. $\mathcal{N}(0,1)$. Let $A \subset \left\{1, \dots, N\right\}$ denote a random set of size $n$ (chosen uniformly at random among all $n-$element subsets of $A$). Then
		$$ \mathbb{E}_{{\textnormal{X},{\omega}}} \lim_{\ell \rightarrow \infty} \left\| \frac{1}{\ell} \sum_{i=1}^{\ell}{ \sum_{a \in A_i}^{}{ \frac{\myvec{g}_a}{\|\myvec{g}_a\|^2} \myvec{\omega}_a} } \right\| \leq \frac{n}{N \sqrt{D-2}}.$$
	\end{lemma}
	\begin{proof}
		We observe that the vectors $\myvec{g}_a$ are, albeit Gaussian random vectors, fixed once given and so are the $\myvec{\omega}_a$. Thus, the law of large numbers implies that averaging over many randomly chosen subsets $A \subset \left\{1, 2, \dots, N\right\}$ of size $A$ results, ultimately, in each coordinate being picked the same number of times and thus
		$$ \lim_{\ell \rightarrow \infty} \frac{1}{\ell} \sum_{i=1}^{\ell}{ \sum_{a \in A_i}^{}{ \frac{\myvec{g}_a}{\|\myvec{g}_a\|^2} \myvec{\omega}_a} } = \frac{n}{N} \sum_{a=1}^{N} \frac{\myvec{g}_a}{\|\myvec{g}_a\|^2} \myvec{\omega}_a.$$
		We have
		$$  \frac{n}{N} \sum_{a=1}^{N} \frac{\myvec{g}_a}{\|\myvec{g}_a\|^2} \myvec{\omega}_a = \frac{n}{N}  \sum_{a=1}^{N} \frac{\myvec{g}_a}{\|\myvec{g}_a\|} \frac{\myvec{\omega}_a}{\|\myvec{g}_a\|}.$$
		We interpret this as follows: the vector $\myvec{g}_a/\|\myvec{g}_a\|$ is uniformly distributed over the unit sphere in $\mathbb{R}^D$ (a consequence of the
		radial symmetry of the Gaussian distribution), the vector $\myvec{\omega}^* = (\myvec{\omega}_a/\|\myvec{g}_a\|)_{a=1}^{N}$ is interpreted as a random vector. Again, as a consequence
		of the radial symmetry, the vector $\myvec{g}_a/\|\myvec{g}_a\|$ and the size $\|\myvec{g}_a\|$ can be interpreted as independent random variables.
		We compute $\mathbb{E}_{{\textnormal{X},{\omega}}} \left\|  \sum_{a=1}^{N} \frac{\myvec{g}_a}{\|\myvec{g}_a\|} \frac{\myvec{\omega}_a}{\|\myvec{g}_a\|}\right\|^2 $ as
		\begin{align*}
			\sum_{a_1, a_2 = 1}^{N} \mathbb{E}_{{\textnormal{X},{\omega}}} \left\langle  \frac{g_{a_1}}{\|g_{a_1}\|} \frac{\myvec{\omega}_{a_1}}{\|\myvec{g}_{a_1}\|}, \frac{\myvec{g}_{a_2}}{\|\myvec{g}_{a_2}\|} \frac{\myvec{\omega}_{a_2}}{\|\myvec{g}_{a_2}\|} \right\rangle\\
			= \sum_{a = 1}^{N} \mathbb{E}_{{\textnormal{X},{\omega}}} \frac{\myvec{\omega}_{a}^2}{\|\myvec{g}_{a}\|^2}  
			=  \sum_{a = 1}^{N} \mathbb{E}_{{\textnormal{X}}} \frac{1}{\|\myvec{g}_{a}\|^2} =  \mathbb{E}_{} \frac{N}{\|\myvec{g}_{}\|^2}.
		\end{align*}
		This quantity is the mean of an inverse $\chi-$distribution which is $1/(D-2)$ for $D \geq 3$. Thus, using the Cauchy-Schwarz inequality, we get
		$$  \mathbb{E}_{{\textnormal{X},{\omega}}} \frac{n}{N} \left\|  \sum_{a=1}^{N} \frac{\myvec{g}_a}{\|\myvec{g}_a\|} \frac{\myvec{\omega}_a}{\|\myvec{g}_a\|}\right\| \leq \frac{n}{N} \frac{\sqrt{N}}{ \sqrt{D-2}} = \frac{n}{\sqrt{N} \sqrt{D-2}}.$$
	\end{proof}

	\section{Acknowledgement}
	The authors would like to thank Anna Gilbert and Holger Rauhut for useful discussions. %The work was funded by NSF DMS-1763179 and the Alfred P. Sloan Foundation.
	
	\bibliographystyle{IEEEtran}
	\bibliography{references}

\end{document}